\documentclass[12pt]{article}

\usepackage{latexsym}
\usepackage{amssymb}
\usepackage{amsmath}
\usepackage{amsthm}
\usepackage{amsfonts}
\usepackage{latexsym}
\usepackage{color}
\newtheorem{theorem}{Theorem}[section]

\newtheorem{lemma}[theorem]{Lemma}
\newtheorem{corollary}[theorem]{Corollary}

\newcommand{\Aut}{\hbox{Aut}\, }

\newcommand{\'}{\' \i}

 \def\s{\sigma}

\def\di{\bigm|} \def\lg{\langle} \def\rg{\rangle}
\def\nd{\mathrel{\bigm|\kern-.7em/}}

\def\f{\noindent}

\def\Aut{\hbox{\rm Aut}}

\def\mod{\hbox{\rm mod }}

\begin{document}
\begin{center}
{\Large $2$-Groups that factorise as products of cyclic groups,\\
and regular embeddings of complete bipartite graphs}
\end{center}

\vskip 5mm {\small
\begin{center}
{\sc Shaofei Du} \\ {\small School of Mathematical Sciences, Capital
Normal University,}\\ {\small Beijing 100048, China}\\

\vskip 3mm {\sc Gareth Jones}\\
 {\small School of  Mathematics, University of Southampton,}
\\ {\small Southampton S017 1BJ, United Kingdom}\\

\vskip 3mm {\sc Jin Ho Kwak} \\
{\small Department of Mathematics,} \\ {\small Pohang University
of Science and Technology,} \\ {\small Pohang 790-784,
Korea}\\

\vskip 3mm {\sc Roman Nedela}\\
 {\small Institute of Mathematics, Slovak Academy of Science,} \\
 {\small Severn\' a 5, 975 49 Bansk\' a Bystrica, Slovakia}\\

\vskip 3mm {\sc Martin \v Skoviera}\\
 {\small Department of Computer Science,
 Comenius University,} \\
{842 48 Bratislava, Slovakia}
\end{center}}

\begin{abstract}

We classify those 2-groups $G$ which factorise as a product of two disjoint cyclic subgroups $A$ and $B$, transposed by an automorphism of order $2$. The case where $G$ is metacyclic having been dealt with elsewhere,
we show that for each $e\ge 3$ there are exactly three such non-metacyclic groups $G$ with $|A|=|B|=2^e$, and for $e=2$ there is one.
These groups appear in a classification by Berkovich and Janko of $2$-groups with one non-metacyclic maximal subgroup; we enumerate these groups, give simpler presentations for them, and determine their automorphism groups.

\end{abstract}



\newpage

\section{Introduction}\label{sec:intro}

Groups that factorise as products of isomorphic cyclic groups
have been studied for over fifty years
\cite{D,Huppap,Ito1,Ito2, IO}. In several recent papers
\cite{DJKNS1,DJKNS2,J,JNS2,JNS1,KK1,KK2,NSZ} these groups have
emerged as an important tool for the classification of regular
embeddings of complete bipartite graphs in orientable
surfaces. They also arise naturally in the theory of finite $p$-groups, for example in the recent classification by Berkovich and Janko~\cite[Chapter 87]{BJ} of
$2$-groups with a unique non-metacyclic maximal subgroup. Our aim in this paper is to demonstrate some connections between these two problems by showing that a certain class of non-metacyclic $2$-groups play an important role in both situations. As a consequence, we are able to give more information and simpler presentations for some of the groups described by Berkovich and Janko.

As shown in \cite{JNS1}, the problem of
classifying orientably regular embeddings of complete
bipartite graphs $K_{n,n}$ is closely related to that of determining those
groups $G$ that factorise as a product $ AB$ of two cyclic
groups $A=\lg a\rg$ and $B=\lg b\rg$ of order $n$ such that
$A\cap B=1$ and there is an automorphism of $G$ transposing
the generators $a$ and~$b$. Such groups are called
\textit{isobicyclic}, or \textit{$n$-isobicyclic} if we wish to
specify the value of $n$ (see \cite{JNS1}). We will call $(G,a,b)$ an \textit{isobicyclic triple}, and $a, b$ an \textit{isobicyclic pair\/} for $G$.

A result of It\^o~\cite{Ito1} shows that an isobicyclic group $G$, as a product of two abelian groups, must be metabelian. In particular it is solvable, so it satisfies Hall's Theorems, which extend Sylow's Theorems from single primes to sets of primes. This fact, together with results of Wielandt~\cite{Wie} on products of nilpotent groups, allows one to reduce the classification of $n$-isobicyclic groups to the case where $n$ is a prime power (see~\cite{J} for full details).

When $n$ is an odd prime power, a result of Huppert
\cite{Huppap} implies that $G$ must be metacyclic.
When $n$ is a power of $2$, however, Huppert's result does not
apply, and indeed for each $n=2^e\geq 4$ there are non-metacyclic $n$-isobicyclic groups. In this paper we will study all
$n$-isobicyclic groups where $n=2^e$, our main goal being to give a complete description of the corresponding isobicyclic triples $(G,a,b)$. 

In order to state our main result, let us define
\begin{equation}
G_1(e,f)=\lg h, g\di h^{2^e}=g^{2^e}=1,\, h^g=h^{1+2^f}\rg 
\end{equation}
\f where $f=2, \dots, e$, and
\begin{equation}
\begin{array}{ll}
G_2(e;k,l)=\lg  a, b\di&  a^n=b^n=[b^2, a^2]=1,\,
[b,a]=a^2b^{-2}(a^{n/2}b^{n/2})^k,\\
& \hskip-1.5mm (b^2)^a=b^{-2}(a^{n/2}b^{n/2})^l,\,
(a^2)^b=a^{-2}(a^{n/2}b^{n/2})^l\rg \end{array}
\end{equation}
where $n=2^e\ge 4$
and $k,l\in\{0,1\}$, with $k=l=0$ when $n=4$. In fact, it is easily seen that this last group $G_2(2;0,0)$ has a simplified presentation
\begin{equation}
G_2(2;0,0)=\lg a, b\di a^4=b^4=[a^2, b]=[b^2, a]=1,\, [b,a]=a^2b^2 \rg.
\end{equation}

Our main result shows that if $n=2^e$ then every $n$-isobicyclic group has one of the above two forms.

\begin{theorem}\label{main1}
Let $G$ be an $n$-isobicyclic group where $n=2^e\ge 4$. Then
either
\begin{enumerate}
\item[{\rm (i)}] $G$ is metacyclic, and $G\cong G_1(e,f)$
    where $f=2, \dots, e$, or

\item[{\rm (ii)}] $G$ is not metacyclic, in which case either $G\cong
    G_2(2;0,0)$, or $e\ge 3$ and $G\cong G_2(e; k,l)$ where
    $k,l\in\{0, 1\}$. In the latter case there are, up to isomorphism, just three groups for each $e$, with $G_2(e;0,1)\cong
    G_2(e;1,1)$.
\end{enumerate}
\end{theorem}

The metacyclic groups $G_1(e,f)$ were treated in detail in \cite{DJKNS1}; for instance, it was shown there that, up to automorphisms of $G$, one can take the isobicyclic pair to have the form $a=g^r$ and $b=g^rh$, where $r$ is an odd integer such that $1\leq r\leq 2^{e-f}$. This paper is therefore devoted to the non-metacyclic groups $G_2(e; k,l)$.

These groups $G_2(e; k,l)$ have recently arisen in a purely group-theoretic
context. In~\cite[Chapter 87]{BJ} Berkovich and Janko, having classified the minimal non-metacyclic $2$-groups (i.e.~those with all their maximal subgroups metacyclic), then classify those
$2$-groups with a unique non-metacyclic maximal subgroup. Clearly such a group requires at most three generators (two to generate a metacyclic maximal subgroup, and one more outside it). The $3$-generator groups of this type are relatively easy to deal with, and Berkovich and Janko devote most of their analysis to the $2$-generator groups. In Corollary~87.13 they show that any such group factorises as a product of two cyclic groups, and conversely in Theorem 87.22 they show that any non-metacyclic group which factorises in this way (and is therefore a $2$-generator group) has a unique non-metacyclic maximal subgroup. Their analysis of the $2$-generator groups depends on considering the different possibilities for the commutator subgroup, and one part of the classification (essentially Theorem~87.19, see also~\cite[Theorem 4.11]{Jan}) is as follows:

\begin{theorem}\label{BJ87.19}~{\rm(Berkovich and Janko)}
Let G be a $2$-generator $2$-group with exactly one
non-metacyclic maximal subgroup. Assume that $G'\cong
C_{2^r}\times C_{2^{r+1}}$ where $r\ge 2$. Then
\begin{equation}
\begin{array}{ll}
G=\lg a,x \mid
    &a^{2^{r+2}}=1,\ [a,x]=v,\ [v,a]=b,\ v^{2^{r+1}}=b^{2^r}=[v,b]=1,\\
    &v^{2^r}=z,\ b^{2^{r-1}}=u,\ x^2\in\lg u,z\rg\cong C_2\times C_2,\ b^x=b^{-1},\\
    &v^x=v^{-1},\ b^a=b^{-1},\ a^4=v^{-2}b^{-1}w,\ w\in\lg u,z\rg\, \rg
\end{array}
\end{equation}
with $|G|=2^{2r+4}$ and $G'=\lg b\rg \times \lg v\rg \cong C_{2^r}\times C_{2^{r+1}}$. 
\end{theorem}

One should regard~(4) as giving sixteen presentations for each $r$, since there are four possibilities for each of $x^2$ and $w$ in the Klein four-group $\lg u, z\rg$. In Theorem~\ref{isomorphism}, we will show that the groups $G_2(e; k,l)$ for $e\ge 4$ are exactly those groups $G$ in Theorem~\ref{BJ87.19} for which $x^2=z^k$ and $w=z^l$ for some $k, l\in\{0, 1\}$, with $e=r+2$. As noted by Janko in~\cite[p.~315]{Jan}, the classification problem is not completely solved since some pairs of presentations define isomorphic groups. Indeed Theorem~\ref{main1} shows that for each $r\ge 2$ there are, up to isomorphism, just three groups presented by~(4) with $x^2=z^k$ and $w=z^l$, those with $l=1$ and $k=0, 1$ being isomorphic to each other. As a consequence of Theorem~4.2, in~(2) we give slightly more transparent presentations for these groups, showing that each is an extension of its Frattini subgroup $\Phi(G)\cong C_{2^{r+1}}\times C_{2^{r+1}}$ by $C_2\times C_2$: the roles of $a, b$ and $a^{n/2}b^{n/2}$ in~(2) are played by $a, ax$ and the central involution $z$ in~(4). Moreover, all our structural results proved in Section~\ref{sec:g2} for the groups $G_2(e;k,l)$ apply to these groups $G$. For instance, we show that they are all metabelian, of exponent $2^e$ and nilpotence class $e$. In classifying all isomorphisms between the groups $G_2(e;k,l)$, we also determine their automorphisms; in particular, we show that for each $e\ge 3$, ${\rm Aut}\,G_2(e;k,l)$ has order $2^{4e-3}$ or $2^{4e-4}$ as $l=0$ or $1$.

Section~\ref{sec:deri} begins a 
structural analysis of isobicyclic $2$-groups in general, while
Section~\ref{sec:nonmeta} is devoted specifically to
non-metacyclic isobicyclic groups $G$. We show that if $n=2^e$ then either $G$ has a cyclic derived subgroup, in which case $e=2$ and $G\cong G_2(2;0,0)$, or $G$ has a derived group generated by
two elements, in which case $e\ge 3$ and $G$ is isomorphic to one of the three non-isomorphic groups of the form $G_2(e; k,l)$. This proves part~(ii) of Theorem~\ref{main1}, and since part~(i) is dealt with in~\cite{DJKNS1}, it completes the proof of that theorem. In Section~5 we apply results from the
preceding sections to the classification of regular embeddings
of complete bipartite graphs $K_{n,n}$ where $n$ is a power
of $2$.

A completely different proof of  Theorem~\ref{main1}(ii) has already been given in~\cite{DJKNS2}; it proceeds by induction on $e$, based on the fact that if $n=2^e$ then any $n$-isobicyclic group has an $m$-isobicyclic quotient where $m=2^{e-1}$. However, the main purpose of that paper was not to study these groups for their own sake, but rather to enumerate them and to obtain sufficient information about them to determine the corresponding graph embeddings. Here we present an alternative proof, designed to shed more light on the internal structure of these groups, and on how they are related to more general classes of $2$-groups.

\vskip3mm

\section{Non-metacyclic groups $G_2(e;k,l)$}\label{sec:g2}

In this section we analyse properties of the non-metacyclic
groups $G_2(e;k,l)$ appearing in Theorem~\ref{main1}.
Throughout this section we write $G(k,l)$, or simply $G$, instead of $G_2(e; l,k)$. For brevity we also write $n=2^e$ and $m=n/2=2^{e-1}$.

It is useful to note that each group $G$ has a Frattini subgroup $\Phi=\Phi(G)=\lg a^2\rg\times\lg b^2\rg\cong C_m\times C_m$, with $G/\Phi\cong C_2\times C_2$ (see~\cite[Prop.~2.1]{DJKNS2}). It therefore has three maximal subgroups, namely $\lg\Phi,a\rg=\Phi\cup\Phi a$,  $\lg\Phi,b\rg=\Phi\cup\Phi b$ and  $\lg\Phi,ab\rg=\Phi\cup\Phi ab$.

\begin{lemma} \label{G3} The following properties hold in $G=G(k,l)$.
\begin{enumerate}

\item[{\rm (i)}] The elements $a^m$, $b^m$, and $z=a^mb^m$
    are central involutions of $G$.

\item[{\rm (ii)}]
$$b^ja^i=\left\{
\begin{array}{lll}& a^ib^j, \quad &{\rm for\ }\,  i {\rm \ and\ } j {\rm\  \, even },\\
&a^ib^{-j}z^{lj/2}, \quad &{\rm for\ }\,  i {\rm \ odd\, and\ } j {\rm\ even},\\
&a^{-i}b^jz^{li/2}, \quad &{\rm for\ }\,  i {\rm \ even\, and\ } j {\rm\ odd },\\
&a^{-i}b^{-j}z^{k+l(i+j)/2}, \quad &{\rm for\ }\,  i {\rm \ and\ } j {\rm\ odd }.\\
\end{array}\right.$$

\item[{\rm (iii)}]
The element $g=a^ib^j$ has order
$$|g|\;\left\{
\begin{array}{lll}
& {\rm dividing\ } m, \quad &{\rm for\ }\,  i {\rm \ and\ } j {\rm\  \, even },\\
&=n, {\rm \ with\ } g^m=a^m,\quad &{\rm for\ }\,  i {\rm \ odd\, and\ } j {\rm\ even},\\
&=n, {\rm \ with\ } g^m=b^m,\quad &{\rm for\ }\,  i {\rm \ even\, and\ } j {\rm\ odd},\\
& {\rm dividing\ } 4, =2 {\rm \ if\ }k=l=0,\quad &{\rm for\ }\,  i {\rm \ and\ } j {\rm\  \, odd }.\\
\end{array}\right.$$

\item[{\rm (iv)}] The group $G$ is isobicyclic, that is,
    $G=\lg a\rg \lg b\rg ,$ where $|a|=|b|=2^e$ and $\lg
    a\rg \cap \lg b\rg =1$, and there is an involutory
    automorphism of $G$ interchanging $a$ and $b$.
    
\item[{\rm (v)}] $G'=\lg a^2b^{-2}z^k\rg \times \lg
    a^4z^l \rg$ with $\lg a^2b^{-2}z^k\rg\cong C_m$ and $\lg
    a^4z^l \rg \cong C_{m/2}$.
 
\item[{\rm (vi)}] $G$ is not metacyclic.

\item[{\rm (vii)}] $G$ has nilpotence class $e$, with upper central series $1=Z_0<Z_1<\cdots <Z_e=G$ where $Z_i=\lg a^{2^{e-i}}\rg\lg b^{2^{e-i}}\rg$ for $i=0, 1, \ldots, e$.

\end{enumerate}
\end{lemma}

\begin{proof}
If we define $z=a^mb^m$, the defining
relations for $G$ in~(2) take the form
\begin{equation}
a^n=b^n=[b^2, a^2]=1,\ [b,a]=a^2b^{-2}z^k,\
(b^2)^a=b^{-2}z^l,\ (a^2)^b=a^{-2}z^l,
\end{equation}
\f where $k, l\in\{0, 1\}$.

\vskip 3mm (i) Since $m$ is even, and $[a^2,b^2]=1$, the involutions $a^m$ and $b^m$ commute; they are distinct, so their product $z$ is also an involution. Since $z$ commutes with $a^2$ and $b^2$, and $m$ is divisible by $4$, we have
$(b^{m})^a=((b^2)^a)^{m/2}=(b^{-2}z^l)^{m/2}=b^{-m}z^{lm/2}=b^{-m}=b^m$,
so $b^{m}$ is an element of the centre $Z(G)$ of $G$. Similarly, $a^{m}\in Z(G)$, so $z\in Z(G)$.

\vskip 3mm (ii) Now we compute $b^ja^i$. Define
$c=[b,a]=b^{-1}a^{-1}ba$. If both $i$  and $j$ are even, then
$b^ia^j=a^jb^i$. If $i$ is odd  and $j$ is even, then since $i-1$ is even we have
$$b^ja^i=(b^ja^{i-1})a=a^{i-1}b^ja=a^i((b^2)^a)^{j/2}=a^ib^{-j}z^{lj/2}.$$
If $i$ is even  and $j$ is odd, then
$$b^ja^i=b^{-1}(b^{j+1}a^i)=(b^{-1}a^ib)b^j=((a^2)^b)^{
i/2}b^j=a^{-i}b^jz^{li/2}.$$
If both $i$ and $j$ are odd, then
$$\begin{array}{lll}
b^ja^i&=&b^{j-1}abca^{i-1}=b^{j-1}ab(b^{-2}a^2z^k)a^{i-1}=
a(b^{j-1})^a(a^{i+1})^bb^{-1}z^k\\&=&ab^{1-j}z^{l(j-1)/2}
a^{-i-1}z^{l(i+1)/2}b^{-1}z^k=a^{-i}b^{-j}z^{k+l(i+j)/2}.
\end{array}$$

\vskip 3mm (iii) If $i$ and $j$ are both even then (ii) implies that
$$g^2=(a^ib^j)^2=a^i(b^ja^i)b^j=a^{2i}b^{2j}\in\lg a^4\rg\times\lg b^4\rg\cong C_{m/2}\times C_{m/2},$$
so $g^m=1$. If $i$ is odd and $j$ is even then (ii) gives
$$g^2=a^i(b^ja^i)b^j=a^{2i}b^{-j}z^{lj/2}b^j=a^{2i}z^{lj/2},$$
so $g^4=a^4$; thus $g^{2^r}=a^{2^r}$ for all $r\ge 2$, so $|g|=|a|=n$ with $g^m=a^m$. The proofs in the other two cases are similar.

\vskip 3mm (iv) The  formul\ae\/ in (ii) show that every element
of $G$ can be expressed in the form $a^ib^j$, so $G=\lg a\rg
\lg b\rg$.
In order to see that $\lg a\rg \cap \lg b\rg =1$, note that $a^i$ and $b^j$ lie in distinct cosets of $\Phi$ unless $i$ and $j$ are both even; in this case the fact that $\Phi=\lg a^2\rg\times\lg b^2\rg$ ensures that $\lg a\rg\cap\lg b\rg=1$. The defining relations of $G$ are equivalent to those obtained by transposing $a$ and $b$, so this transposition
can be
extended an automorphism $\alpha$ of order $2$ of $G$. Hence $G$ is an $n$-isobicyclic group.

\vskip 3mm (v) Since $G=\lg a, b\rg$, $G'$ is the normal closure $\lg c^g\mid g\in G\rg$ in $G$ of the commutator $c=[b,a]$. We will show that this is the subgroup $M:=\lg c, c^a\rg $. Since
$c=[b,a]=a^2b^{-2}z^k$ we have
$c^a=(a^2b^{-2}z^k)^a=a^2b^2z^{k+l}$, and conjugation by $a$ transposes these two generators of $M$ since $[c,a^2]=1$. Similarly, conjugation by $b$ transposes the generators $c$ and $(c^a)^{-1}$ of $M$, so $M$ is normal in $G$ and hence $M=\lg c^g\mid g\in G\rg=G'$. Thus $G'$ has generators $c=a^2b^{-2}z^k=a^{km+2}b^{km-2}$ and $c^ac=a^4z^l=a^{lm+4}b^{lm}$; these generate disjoint cyclic groups of orders $m$ and $m/2$, so
$$G'=\lg a^2b^{-2}z^k\rg \times \lg a^4z^l\rg \cong C_m \times C_{m/2}.$$

\vskip 3mm (vi) For $e \ge 3$ the fact that $G'$ is not cyclic immediately implies that $G$ is not metacyclic. In the case $e=2$ it is easily seen that the only cyclic normal subgroups of $G$ are contained in $\Phi$, and these do not have cyclic quotients.

\vskip 3mm (vii) This follows by induction on $e$, using the facts that $Z(G)=\{1, a^m, b^m, z\}$ (a simple consequence of (ii)), that $G/Z(G)\cong G(e-1;0,0)$, and that $G(2;0,0)$, as presented in (3), clearly has class $2$.
\end{proof}

\begin{lemma}\label{auto}
Each isomorphism  $\s\colon G(k_1,l_1)\to G(k,l)$ is
given by setting $a_1^\s =a^ib^j$ and $b_1^\s =a^fb^h$, where
\begin{enumerate}
\item[{\rm (i)}] $k_1\equiv k+\frac{l(f+h-i-j)}{2} \pmod{2},$
\item[{\rm (ii)}] $l_1=l$, and
\item[{\rm (iii)}] either $i$ and $h$ are odd and $j$ and $f$ are even, or $i$ and $h$ are even and $j$ and $f$ are
odd.
\end{enumerate} Moreover, each choice of the parameters $i,j,f$ and $h$
satisfying the above conditions determines an isomorphism
$G(k_1,l_1)\to G(k,l)$.
\end{lemma}

\begin{proof}
Recall that
$$\begin{array}{ll}
G=G(k,l)=\lg a, b \di& a^n=b^n=[b^2, a^2]=1,\,
[b,a]=a^2b^{-2}z^k,\\
&(b^2)^a=b^{-2}z^l,\,
(a^2)^b=a^{-2}z^l\rg ,
\end{array}$$
and define
$$\begin{array}{ll}
G_1=G(k_1,l_1)=\lg a_1, b_1 \di& a_1^n=b_1^n=[b_1^2, a_1^2]=1,\,
[b_1,a_1]=a_1^2b_1^{-2}z_1^{k_1},\\
& (b_1^2)^{a_1}=b_1^{-2}z_1^{l_1},\,
(a_1^2)^{b_1}=a_1^{-2}z_1^{l_1}\rg,
\end{array}$$
where $z_1=a_1^mb_1^m$.

An isomorphism $\sigma\colon G_1\to G$ is uniquely determined by an assignment
$$a_1\mapsto a_2 = a^ib^j,\qquad b_1\mapsto b_2 = a^fb^h$$
for some integers $i$, $j$, $f$ and $h$ such that $a_2$ and $b_2$ generate $G$ and satisfy the defining relations of $G_1$, when substituted for $a_1$ and $b_1$.

Now $a_1$ has order $n$, whereas Lemma~\ref{G3}(iii) shows that $a_2$ has order less than $n$ if $i$ and $j$ are both even or both odd. We may therefore restrict attention to mappings $\s$ for which $i$ and $j$ have opposite parity, that is, $a_2\in\Phi a\cup\Phi b$. A similar argument shows that $b_2\in\Phi a\cup\Phi b$. If $a_2$ and $b_2$ are both in $\Phi a$, or both in $\Phi b$, they are both contained in a maximal subgroup $\Phi\cup\Phi a$ or $\Phi\cup\Phi b$ of $G$ and hence cannot generate $G$. They therefore lie in distinct cosets $\Phi a$ and $\Phi b$, and by composing $\s$ with the automorphism $\alpha$ of $G$ transposing $a$ and $b$ if necessary, we may assume that $a_2\in\Phi a$ and $b_2\in\Phi b$, that is, $i$ and $h$ are odd while $j$ and $f$ are even. This ensures that $a_2$ and $b_2$ generate $G$, since none of the three maximal subgroups of $G$ contains both of them.

For any $g\in G$ we have $g^2\in\Phi\cong C_m\times C_m$, so $a_2$ and $b_2$ satisfy the first three relations $a_2^n=b_2^n=[b_2^2,a_2^2]=1$ for $G_1$.

Now $\sigma$ sends $z_1=a_1^mb_1^m$ to $a_2^mb_2^m$. Since $i$ is odd and $j$ is even, we have $a_2^m=a^m$ by Lemma~\ref{G3}(iii).
Similarly $b_2^m=b^m$, so $\sigma$ sends $z_1$ to $a^mb^m=z$.

We can now consider the fourth relation. Straightforward calculations give
$$[b_2,a_2]=[a^fb^h,a^ib^j]=a^{2i}b^{-2h}z^{k+l(h-i)/2}$$
and
$$a_2^2b_2^{-2}z^{k_1}=a^{2i}b^{-2h}z^{k_1+l(j-f)/2},$$
so we require
$$k_1\equiv k+\frac{l(f+h-i-j)}{2} \pmod{2},$$
giving condition~(i) of the Lemma.

For the fifth relation, we have 
$$(b_2^2)^{a_2}=((a^fb^h)^2)^{a^ib^j}=((b^{2h}z^{lf/2})^{a^ib^j}=b^{-2h}z^{l+lf/2}$$
and
$$b_2^{-2}z^{l_1}=b^{-2h}z^{l_1-lf/2};$$
since $f$ is even and $h$ is odd we require $l_1=l.$
Similar arguments show that the sixth and final relation is also equivalent to this, so we have condition~(ii).

Conditions~(i) and (ii) are necessary and sufficient conditions for $\sigma$ to be an isomorphism, in the case where $a_2\in\Phi a$ and $b_2\in\Phi b$, that is, $i$ and $h$ are even while $j$ and $f$ are odd. For the case where $a_2\in\Phi b$ and $b_2\in \Phi a$ we can compose $\sigma$ with $\alpha$, transposing $i$ with $j$, and $f$ with $h$; this gives condition (iii) of the Lemma, leaving conditions (i) and (ii) unchanged.
\end{proof}

\begin{corollary}\label{iso}
For each $e\ge 3$ we have $G(1, 1)\cong G(0, 1)$ while  $G(0, 0)$, $G(
1, 0)$ and $G(0, 1)$ are pairwise non-isomorphic.
\end{corollary}

\begin{proof}
From Lemma~\ref{auto} we immediately deduce that $G(k, 0) \,\not\cong\, 
G(k', 1)$ for any $k$ and $k'$, and that $G(0,0)\,\not\cong\, G(1,0)$. Furthermore,
taking $i=3$, $j=f=0$ and $h=1$ in the definition of $\s $, we
get an isomorphism from $G(0,1)$ to $G(1,1).$
\end{proof}

\begin{corollary}\label{autos}
The automorphisms of $G(k,l)$ are given by $\s: a\mapsto a^ib^j,\; b\mapsto a^fb^h$ where
\begin{enumerate}
\item[{\rm (i)}] either $i$ and $h$ are odd and $j$ and $f$ are even, or $i$ and $h$ are even and $j$ and $f$ are odd, and
\item[{\rm (ii)}] $i+j\equiv f+h \pmod{4}$ if $l=1$.
\end{enumerate}
\end{corollary}

\begin{proof}
This follows immediately from Lemma~\ref{auto}, with $k_1=k$ and $l_1=l$.
\end{proof}

By counting choices of $i, j, f, h\in{\mathbb Z}_n$ satisfying the conditions of Corollary~\ref{autos}, we deduce that $|{\rm Aut}\,G(k,l)|=n^4/8$ or $n^4/16$ as $l=0$ or $1$.

\section{The derived group of an isobicyclic
$2$-group}\label{sec:deri}

In this section we begin an analysis of the structure of an
isobicyclic $2$-group. Let $(G,a,b)$ be an $n$-isobicyclic triple where $n=2^e\ge 4$. As before, let $c=[b, a]$ and let $\Phi$
denote the Frattini subgroup $\Phi(G)$ of $G$. Let  $\mho_i(G)=\lg g^{2^i}\mid \, g\in G\rg$, and let $K_i(G)=[G, G, \cdots, G]$ ($i$ times); in particular, $K_2(G)=G'$. These are all characteristic subgroups of~$G$.

The following properties of $G$ follow from more general known
results.

\begin{lemma}\label{general}
Let $(G,a,b)$ be a non-abelian $n$-isobicyclic triple where $n=2^e\ge 4$, and let $A=\lg a\rg$ and $B=\lg b\rg$. Then the following hold.
\begin{itemize}
\item[{\rm (i)}] The derived group $G'$ is abelian
    \emph{(see \cite{Ito1})}.

\item[{\rm (ii)}] $G'/(G'\cap A)$ is isomorphic to a
    subgroup of $B$ \emph{(see \cite[Corollary C]{CI})}.

\item[{\rm (iii)}] $G$ is metacyclic if and only if
    $G/\Phi(G')K_3(G)$ is metacyclic \emph{(see~\cite{Bla} or
    \cite[Hilfssatz III.11.3]{Hup})}.
\end{itemize}
\end{lemma}

\begin{lemma} \label{Z_i} {\rm \cite[Lemma 3.1]{DJKNS1}}
Let $G$ be an isobicyclic $2$-group of exponent $2^e\ge 4$. Then
$G$ has a central series $1=Z_0<Z_1<Z_2<\cdots<Z_e=G$ of
subgroups $Z_i=\lg a^{2^{e-i}}\rg\lg b^{2^{e-i}}\rg$ of order~$2^{2i}$. Moreover,
$\mho_i(G)=Z_{e-i}$ and $Z_i/Z_{i-1}\cong C_2\times C_2$ for
each $i\in\{1,2,\ldots, e\}$. In particular, for every element
$g\in Z_i$ we have $|g|\le 2^i$.
\end{lemma}

\begin{proof}
We proved this result as Lemma~3.1 of~\cite{DJKNS1}, so we simply outline the argument here. By a result of Douglas~\cite{D} and It\^ o~\cite{Ito2} (see also~\cite[VI.10.1(a)]{Hup}), the core of $A$ in $G$ is nontrivial. Since $\lg a^{2^{e-1}}\rg$ is the unique minimal normal subgroup of $A$ it is therefore normal in $G$, and hence central. The same applies to $\lg b^{2^{e-1}}\rg$, so these two disjoint subgroups generate a central subgroup $Z_1\cong C_2\times C_2$. Now apply the same argument to the isobicyclic group $G/Z_1$, and iterate.
\end{proof}

\begin{lemma} \label{deri}
Let $(G,a,b)$ be a non-abelian $n$-isobicyclic triple where $n=2^e\ge 4$, and let $A=\lg a\rg$ and $B=\lg b\rg$. Then $G$ has the following properties.
\begin{enumerate}
\item[{\rm (i)}] There exists an odd integer $d<2^e$ and
    integers $u$ and $v$ such that $0\le u<v\le e$,
    $c=[b,a]=a^{d2^u}b^{-d2^u}$ and $G'=\lg c\rg \times \lg
    a^{2^v}\rg =\lg c\rg \times \lg b^{2^v}\rg$. In
    particular, $G'$ is cyclic if $v=e$. Moreover,
    $[a^{2^u}, b^{2^v}]=[b^{2^u}, a^{2^v}]=1$.

\item[{\rm (ii)}] $\lg c\rg \cap A=\lg c\rg \cap B=1$,
    $|c|=2^{e-u}$, and for each integer $j$ such that $0\le
    j\le e-u$ there exists an odd integer $h$ such that
    $c^{2^j}=a^{h2^{u+j}}b^{-h2^{u+j}}$.

\item[{\rm (iii)}] Either $G'=\lg c\rg$, or $G'=\lg c,
    c^a\rg =\lg c, c^b\rg $ with $c^a=c^sa^{t2^v}$ and
    $c^b=c^sb^{-t2^v}$
where  $s$ and $t$ are odd.
\end{enumerate}
\end{lemma}

\begin{proof}
Since $G=AB$, each element can be written as $a^ib^j$, and since
$A\cap B=1$ this representation is unique. Let $c=[b, a]=a^rb^w$.
Since the automorphism $\alpha$ interchanges $a$ and $b$, we have
$$b^ra^w=[b^{\alpha}, a^{\alpha}]=[a, b]=[b,a]^{-1}=b^{-w}a^{-r}.$$
Therefore  $w\equiv -r\ (\mod 2^e)$ and so $c=a^rb^{-r}$. We can write
$r=d2^u$ where $d$ is odd, $d<2^e$ and $0\le u\le e$. Similarly, for every integer $j$
there is an integer $k$ such that $c^j=a^{k}b^{-k}$. In
particular, $\lg c\rg \cap A=\lg c\rg \cap B=1$, as claimed in (ii).

Let the cyclic group $G'\cap A$ be generated by $a^{2^v}$, where $v\le e$. Applying $\alpha$ gives $G'\cap B=\lg
b^{2^v}\rg $. Since $G=\lg a,b\rg$, Lemma~III.1.11
of \cite{Hup} implies that $G'=\lg c^g\di g\in G\rg
$. By Lemma~\ref{general}(i), $G'$ is abelian, so
$c$ is an element of $G'$ of maximal order. Since $\lg c\rg
\cap A=\lg c\rg \cap B=1$, we see that $\lg c\rg \times \lg
a^{2^v}\rg\le G'$. By Lemma~\ref{general}(ii), $G'/\lg
a^{2^v}\rg $ is cyclic. Since $\lg c\rg \cap A=1$ again,  the
image of $c$ in $G'/\lg a^{2^v}\rg $ has order
$|c|$, so it is an element of $G'/\lg a^{2^v}\rg $ of
maximal order and therefore generates $G'/\lg a^{2^v}\rg$. This gives $G'=\lg c\rg \times \lg a^{2^v}\rg$ and hence, by applying $\alpha$,
$G'=\lg c\rg \times \lg b^{2^v}\rg $.

From $a^{2^v}c=ca^{2^v}$ and $c=a^{d2^u}b^{-d2^u}$ we see that $[b^{d2^u},a^{2^v}]=1$. Hence $[\lg b^{d2^u}\rg,a^{2^v}]=1$, and in particular, since $d$ is odd, $[b^{2^u},a^{2^v}]=1$. By symmetry, $[a^{2^u},b^{2^v}]=1$.

Since $\lg c\rg=\lg a^{d2^u}b^{-d2^u}\rg\le \mho_u(G)= Z_{e-u}$, Lemma~\ref{Z_i} shows that $|c|\le 2^{e-u}$.  Since $c$
is an element of maximal order in $G'$, we have
$2^{e-v}=|a^{2^v}|\le |c|\le 2^{e-u}$, so $u\le v$. This proves (i), apart from the inequality $u\ne v$, which follows later.

To prove (ii), let $L$ denote the subgroup  $G'\lg b^{2^u}\rg $. Then
$$L=\lg a^{d2^u}b^{-d2^u},
b^{2^v}\rg \lg b^{2^u}\rg = \lg a^{2^u},  b^{2^u}\rg =\lg
a^{2^u}\rg \lg b^{2^u}\rg =Z_{e-u}.$$ \f Computing the order
$$2^{2(e-u)}=|L|=|G'||\lg b^{2^u}\rg |/|G'\cap \lg b^{2^u}\rg |=|c|
2^{e-v}2^{e-u}/2^{e-v}=|c| 2^{e-u},$$ \f we see that $|c|=2^{e-u}$.

For each $j=0, 1, \ldots, e-u$ we have
$c^{2^j}\in Z_{e-(u+j)}$, so
$c^{2^j}=a^{h2^{u+j}}b^{-h2^{u+j}}$ for some integer $h$. Since
$|c^{2^j}|=2^{e-(u+j)}$, it follows that $c^{2^j}\not\in
Z_{e-(u+j+1)}$, so $h$ is odd. This proves (ii).

We now consider (iii). Since $c^a\in G'=\lg c\rg\times \lg
a^{2^v}\rg$, either $c^a\in\lg c\rg$, or
$c^a=c^sa^{t2^q}$ for some integers $s,t$ and~$q$ where $t$ is
odd and $q\ge v$.
In the former case we have $G'=\lg c\rg$, satisfying (iii); we may therefore assume the latter, in which case we also have $c^b=c^sb^{-t2^q}$. Define $M=\lg c,
c^a\rg$, so $M=\lg c\rg \times \lg a^{2^q}\rg $. From
the preceding paragraph we know that
$c^{2^{q-u}}=a^{h2^q}b^{-h2^q}$ for some odd~$h$. Therefore
$$b^{2^q}\in \lg b^{h2^q}\rg =\lg c^{-2^{lqu}}a^{h2^q}\rg \le
\lg c^{2^{q-u}}, a^{2^q}\rg \le M,$$ which implies that $M=\lg
c\rg \times\lg b^{2^q}\rg.$ Now $M^a=\lg c, a^{2^q}\rg ^a=M$
and $M^b=\lg c, b^{2^q}\rg ^b=M$, so $M^g=M$ for each $g\in
G$. In particular, $c^g\in M$ for each $g\in G$. Therefore
$G'=\lg c^g\di g\in G\rg =M.$ In other words, $q=v$, that is
$c^a=c^sa^{t2^v}$ where $t$ is odd.

We now show that $u\ne v$.
Recall that $G' = \lg c\rg\times\lg
a^{2^v}\rg = \lg c\rg\times\lg b^{2^v}\rg$, so $Z_{e-v}=\lg a^{2^v}\rg\times\lg b^{2^v}\rg\le G'$. On the other hand $c=a^{d2^u}b^{-d2^u}$, so $G'\le\lg a^{2^v}\rg\lg b^{2^v}\rg=Z_{e-u}$. Now suppose that $u=v$, so $G'=Z_{e-v}$, with $e-v>0$ since $G$ is nonabelian.
By Lemma~\ref{Z_i}, the subgroup $G'/Z_{e-(v+1)} = Z_{e-v}/Z_{e-(v+1)}$ is central in $G/Z_{e-(v+1)}$. We have seen that $G'=\lg c,c^a\rg$, so $c^aZ_{e-(v+1)} = cZ_{e-(v+1)}$ since $c\in G'$. Thus
$$G'/Z_{e-(v+1)}=\lg c, c^a\rg /Z_{e-(v+1)}=\lg c\rg /Z_{e-(v+1)}$$
is cyclic, contradicting the fact that
$Z_{e-v}/Z_{e-(v+1)} \cong C_2\times C_2$ by Lemma~\ref{Z_i}. Thus $u<v$, completing the proof of (i).

Finally, since $c^a=c^sa^{t2^v}\in G'=\lg c\rg \times \lg a^{2^v}\rg$ we have
$$2^{e-u}=|c|=|c^a|=|c^sa^{t2^v}|=\max\{ |c^s|, |a^{t2^v}|\},$$
with $|a^{t2^v}|=2^{e-v}<2^{e-u}$ since $v>u$, so $|c^s|=|c|$ and hence $s$ must be odd.
\end{proof}

The next result uses the parameter $u$, where $c$ has order $2^{e-u}$, to distinguish between metacyclic and non-metacyclic $n$-isobicyclic groups $G$.

\begin{lemma}\label{structurenew}
Let $(G,a,b)$ be a non-abelian $n$-isobicyclic triple where $n=2^e\ge 4$. With $u$ and $v$ defined as in Lemma~\ref{deri}, the following
statements hold.
\begin{itemize}
\item[{\rm (i)}] If $u\ge 2$ then $G$ is metacyclic, with $2\le u<v=e$ and $G'=\lg c\rg=\lg (ab^{-1})^r\rg\cong C_{2^{e-u}}$.

\item[{\rm (ii)}] If $u<2$ then $G$ is non-metacyclic, with $u=1$, $v=2$ and $G'=\lg a^2b^{-2}\rg\times\lg a^4\rg=\lg a^2b^{-2}\rg\times\lg b^4\rg$, where $\lg a^2b^{-2}\rg\cong C_{2^{e-1}}$  and $\lg a^4\rg \cong \lg b^4\rg \cong C_{2^{e-2}}$.
    
\end{itemize}
In particular, if $G$ is non-metacyclic and $G'$ is cyclic,
then $e=2$ and $G'=\lg a^2b^{-2}\rg\cong C_2$.
\end{lemma}

\begin{proof}
By Lemma~\ref{deri}, $G'=\lg c\rg\times \lg a^{2^v}\rg$ where
$c=[b,a]=a^rb^{-r}$ with $r=d2^u$ for an odd integer $d$ and some integers $u$ and $v$ such that $0\le u<v\le e$.

(a) We first consider the case where $G'$ is cyclic, so that $v=e$. By Lemma~\ref{deri}, $G'=\lg
c\rg$ and hence $c^a=c^s$ for some $s$, which must be odd. By
applying $\alpha$, which inverts $c$, we also have $c^b=c^s$. It follows that $c^{ab^{-1}}=c$. Moreover, $[c,a]=[c,b]=c^{s-1}\in\lg c^2\rg$,
which means that the image $c\lg c^2\rg$ of $c$ in $G/\lg c^2\rg $ is a central involution in that group. (As a characteristic subgroup of $\lg c\rg=G'$, $\lg c^2\rg$ is normal in $G$.) Now we have
$$
(ab^{-1})^2=ab^{-1}ab^{-1}=a^2(a^{-1}ba)^{-1}b^{-1}
=a^2(bc)^{-1}b^{-1}=a^2c^{-1}b^{-2}\equiv  a^2b^{-2}c\ (\mod \lg c^2\rg )
$$
\f  and
$$a^2b^2=abac^{-1}b=bac^{-1}ac^{-1}b\equiv baab   \equiv
babac^{-1}\equiv b^2ac^{-1}ac^{-1}\equiv b^2a^2\ (\mod \lg c^2\rg
).$$

(a1) Suppose that $u\ge 2$, as in (i). Then $r/2$ is even, so
$$\begin{array}{lll}(ab^{-1})^r&=&((ab^{-1})^2)^{r/2}\equiv
(a^2b^{-2}c)^{r/2}\equiv (a^2b^{-2})^{r/2}c^{r/2}\\
& \equiv & (a^2b^{-2})^{r/2}\equiv a^rb^{-r} \equiv c\ (\mod \lg c^2\rg ).\end{array}
$$
Thus $(ab^{-1})^r$ is an odd power of $c$, so $\lg (ab^{-1})^r\rg=\lg c\rg$ and
$|(ab^{-1})^r|=|c|=2^{e-u}$ by Lemma~\ref{deri}.
Since $G=\lg a,b\rg$, the quotient group $G/G'=G/\lg c\rg$ is generated by the images $\overline a$ and $\overline{ab^{-1}}$ of $a$ and $ab^{-1}$ in this group. Now $\lg c\rg\cap A=1$ by Lemma~\ref{deri}(ii), so $\overline a$ has order $|a|=2^e$. Since $(ab^{-1})^r\in\lg c\rg$, we see that $\overline{ab^{-1}}$ has order dividing $r$, and hence dividing $2^u$. But $G/\lg c\rg$ is an abelian group of order $|G|/|\lg c\rg|=2^{2e}/2^{e-u}=2^{e+u}$, so $\overline{ab^{-1}}$ must have order $2^u$ with $G/G'=\lg \overline a\rg\times\lg\overline{ab^{-1}}\rg\cong C_{2^e}\times C_{2^u}$.

Since $(ab^{-1})^r$ is an odd power of $c$ we have $\lg (ab^{-1})^r\rg = \lg c\rg$, so the cyclic subgroup $H:=\lg ab^{-1}\rg$ contains $G'$ with index $2^u$. Since the image of $H$ in $G/G'$ has order $2^u$, and $G'$ has order $2^{e-u}$, it follows that $H$ has order $2^e$. Since $H$ contains $G'$ it is a normal subgroup of $G$. Thus $AH=HA$ is a subgroup of $G$, and since it contains both $a$ and $b$ we have $G=AH$, so $G$ is metacyclic. This proves (i) in the case where $G'$ is cyclic.

(a2) Now suppose that $G'$ is cyclic and $u=0$. Then $G'=\lg c\rg$ has order $2^{e-u}=2^e$ by Lemma~\ref{deri}(ii). Since $G'\cap A=1$ by Lemma~\ref{deri}(ii) we have $|G'A|=|G'||A|=e^{2e}=|G|$, so $G=G'A$ and $G/G'$ is cyclic. But then $G/\Phi$ is cyclic and hence so is $G$, a contradiction. Hence $u\ne 0$.

We therefore have $u=1$, so $r=2d$, giving $c=a^{2d}b^{-2d}$, where $d$ is odd. By Lemma~\ref{deri}(ii), $|c|=2^{e-1}$, and since $G=\lg a, ab^{-1}\rg$ we have $G/G'=G/\lg c\rg \cong C_{2^e}\times C_2$.
 
(a3) Suppose first that $G$ is metacyclic. Huppert  gives the general form for a metacyclic $p$-group in~\cite[III.11.2]{Hup}; taking $p=2$ we have
$$G=\lg g, h \mid h^{2^i}=1,\ g^{2^j}=h^{2^k},\ h^g=h^q\rg$$
with $0\le k\le i$, $q^{2^j}\equiv 1\pmod{2^i}$ and $2^k(q-1)\equiv 0 \pmod{2^i}$. Thus $G$ has a normal subgroup $H=\lg h\rg\cong C_{2^i}$ with $G/H\cong C_{2^j}$, so $|G|=2^{i+j}$ and hence $i+j=2e$. Since $|h|=2^i$ and $G$ has exponent $n=2^e$ we have $i\le e$ and hence $j\ge e$. Since $|g|=2^{i+j-k}$ we have $i+j-k\le e$ and hence $k\ge e$. But $k\le i$, so $i=j=k=e$. Thus
$$G=\lg g, h \mid g^n=h^n=1,\ h^g=h^q\rg$$
for some $q$. Now $G'$, being cyclic, is generated by $[h,g]=h^{q-1}$. We are assuming that $u=1$, so $G'\cong C_{2^{e-1}}$ and hence $q\equiv 3\pmod{4}$.

Each element of $G$ has the form $g^ih^j$ for a unique pair $i, j\in{\mathbb Z}_n$. By using the relation $(h^J)^{g^i}=h^{jq^i}$, we obtain
$$\begin{array}{ll}
(g^ih^j)^m&=g^{im}(h^j)^{g^{i(m-1)}}(h^j)^{g^{i(m-2)}}\ldots(h^j)^{g^i}h^j\\
&=g^{im}h^{j(q^{i(m-1)+q^{i(m-2)}+\cdots+q^i+1)}}\\
&=g^{im}h^{j(q^{im}-1)/(q^i-1)}
\end{array}$$
for all $m\ge 1$. Let $m=n/2=2^{e-1}$. If $i$ is even then $g^{im}=1$ and $q^i\equiv 1\pmod{4}$; if $2^k\parallel q^i-1$ then $2^{k+e-1}\parallel q^{im}-1$, so $2^{e-1}\parallel (q^{im}-1)/(q^i-1)$ and hence $(g^ih^j)^m=h^{jm}$. If $i$ is odd then $g^{im}=g^m$ and $q^i\equiv 3\pmod{4}$; if $2^k\parallel q^i+1$ (so $k\ge 2$) then $2^{k+e-1}\parallel q^{im}-1$, and $2\parallel q^i-1$, so $2^e\mid (q^{im}-1)/(q^i-1)$ and $(g^ih^j)^m=g^m$. Thus
$$
(g^ih^j)^m=\left\{
\begin{array}{lll}& h^{jm}, \quad &{\rm for\ }\,  i  {\rm\  \, even },\\
&g^m, \quad  &{\rm for\ }\,  i  {\rm\ odd },\\
\end{array}\right.
$$
so $g^ih^j$ has order $n$ if and only if $i$ or $j$ is odd, that is, $g^ih^j\not\in\Phi=\lg g^2, h^2\rg$.

If $a$ and $b$ are an isobicyclic pair for $G$ then they have order $n$, so they are not elements of $\Phi$. Since they generate $G$, they are in different cosets of $\Phi$, namely $g\Phi$, $h\Phi$ or $gh\Phi$. The subgroups $A=\lg a\rg$ and $B=\lg b\rg$ are disjoint, so $a^m\ne b^m$; hence these two cosets cannot be $g\Phi$ and $gh\Phi$ (otherwise $a^m=g^m=b^m$), so one of them must be $h\Phi$, say $a\in h\Phi$. Then $A\Phi=H\Phi$, so $H^{\alpha}\Phi=B\Phi=g\Phi$ or $gh\Phi$, giving $(h^{\alpha})^m=g^m\ne h^m$ and hence $H^{\alpha}\cap H=1$. Since $H$ is a normal subgroup of $G$, so is $H^{\alpha}$. Hence $G=H^{\alpha}\times H$, which is abelian, contradicting the assumption that $G'\cong C_{2^{e-1}}$. Thus $G$ cannot be metacyclic.

(a4) Now suppose that $G$ is non-metacyclic, with $G'$ cyclic and $u=1$ as before. We consider the subgroup $N:=\lg c,ab^{-1}\rg $ of $G$; this is abelian since $c^{ab^{-1}}=c$, and it is normal in $G$ since it contains $G'=\lg c\rg$. Note that $N$ is the preimage in $G$ of
$\lg\overline{ab^{-1}}\rg\le G/G'$. In the abelian group $G/G' = G/\lg c\rg $ we have
$(\overline{ab^{-1}})^{2d}=\overline {a^{2d}b^{-2d}}=\overline
c=\overline 1$, which means that $\overline{ab^{-1}}$ is of
order 2. Since $|c|=2^{e-1}$, we have $|N|=2^e$. Since $\lg
N, a\rg =\lg a, b\rg =G$, we deduce that $G=N\rtimes \lg a\rg
$. Since $G$ is not metacyclic, $N$ can not be cyclic and so
$N\cong C_{2^{e-1}}\times C_2.$  Let $c'$ be an involution of
$N$ different from $c^{2^{e-2}}$, so that $N=\lg c\rg\times\lg
c'\rg $. Then the conjugacy action of $a$ on $N$ is defined by
$c^a=c^s$ and $(c')^a=c^{j2^{e-2}}c'$, where $j=0$ or 1, and
$s$ is odd. Now $G=\lg c, c', a\rg$ with $[c,c']=1$, so $G'=\lg [c,a]^g, [c',a]^g\di g\in G\rg
\le \lg c^2, c^{j2^{e-2}}\rg $. Since $G'=\lg c\rg $, we see
that $j2^{e-2}$ must be odd, so $j=1$ and $e=2$, giving $|G|=16$. Since $v=e$ we have $v=2$, and we have proved (i) in the case where $G'$ is cyclic. (Note that $a^4=b^4=1$ in this case.)

(b) We now consider the case where $G'$ is not cyclic, that is, $v<e$. This immediately implies that $G$ is not metacyclic.
Recall that $u<v$. Since  $G'=\lg c\rg \times \lg a^{2^v}\rg $,
we have $\Phi(G')=\lg c^2\rg $$\times \lg a^{2^{v+1}}\rg $.
Moreover, by Lemma~\ref{deri}(iii) we have $c^a=c^sa^{t2^v}$ for
some odd integers $s$ and $t$. Then $a^{t2^v}=c^{-s+1}[c, a]\in
L:=\Phi (G')K_3(G)$, which implies that $a^{2^v}\in L$.
Since $L$ is a characteristic subgroup of $G$ it also contains $b^{2^v}=(a^{2^v})^{\alpha}$, and hence it contains the subgroup
$Z_{e-v}=\lg a^{2^{e-v}}\rg\lg b^{2^{e-v}}\rg$.

Suppose that $G/Z_{e-v}$ is metacyclic. Since $G/L\cong
(G/Z_{e-v})/(L/Z_{e-v})$, it follows that $G/L$ is metacyclic. Then
Lemma~\ref{general}(iii) implies that $G$ is metacyclic, which is a contradiction. Therefore $G/Z_{e-v}$ is non-metacyclic.

Now $G/Z_{e-v}$ is an isobicyclic $2$-group. Since it is non-metacyclic, and its derived group $(G/Z_{e-v})'=G'/Z_{e-v}\cong C_{2^{v-u}}$ is cyclic, it follows from part~(a2) of this proof that  $G/Z_{e-v}$
has order $16$, with
$|G'/Z_{e-v}|=2$. Thus $a^4\in Z_{e-v}=\lg a^{2^v}\rg\lg
b^{2^v}\rg$, so $v=2$. Since $G'/Z_{e-v}\cong C_{2^{v-u}}$
and $|G'/Z_{e-v}|=2$, we deduce that $v-u=1$, so
$u=1$. We have $G' = \lg c\rg \times \lg a^{2^v}\rg = \lg c\rg \times \lg b^{2^v}\rg$ with $c=a^{d2^u}b^{-d2^u}=a^{2d}b^{-2d}$ for some odd $d$, so $G' = \lg a^2b^{-2}\rg \times \lg a^4\rg = \lg a^2b^{-2}\rg \times \lg b^4\rg$, with first and second factors cyclic of orders $2^{e-1}$ and $2^{e-2}$ as required for (ii).

The final statement in the Lemma is an immediate consequence of (i) and (ii).
\end{proof}

\section{Non-metacyclic isobicyclic
$2$-groups}\label{sec:nonmeta}

The following theorem characterises non-metacyclic isobicyclic
$2$-groups.

\begin{theorem}\label{thm:nonmetachar}
Let $(G,a,b)$ be a non-metacyclic
$n$-isobicyclic triple with $n=2^e\ge 4$.

\noindent{\rm(i)} If $e=2$ then
$$\begin{array}{ll}
G&=\lg a, b\di a^4=b^4=[a^2, b]=[b^2, a]=1, [b,a]=a^2b^2 \rg\\
&\\
&\cong G_2(2;0,0).
\end{array}$$
\noindent{\rm(ii)} If $e\ge 3$ then
$$\begin{array}{ll}
G &= \lg  a, b\di  a^n=b^n=[b^2, a^2]=1,\,
[b,a]=a^2b^{-2}(a^{n/2}b^{n/2})^k,\\
& \hskip 15mm (b^2)^a=b^{-2}(a^{n/2}b^{n/2})^l,\,
(a^2)^b=a^{-2}(a^{n/2}b^{n/2})^l\rg\\
&\\
&\cong G_2(e;k,l) \end{array}$$
where $k,l\in\{0,1\}$.
\end{theorem}

\begin{proof}
By Lemma~\ref{structurenew}(ii) we see that  $v=2$ and $u=1$, so
$$G'=\lg c\rg\times\lg a^4\rg=\lg a^2b^{-2}\rg \times \lg a^4\rg \cong
C_{2^{e-1}}\times C_{2^{e-2}}$$
where $c=[b,a]=a^{2d}b^{-2d}$ for some
odd $d$. By Lemma~\ref{deri}(iii) we have
\begin{equation}
c^a=c^sa^{4t} \quad\text{and}\quad c^b=c^sb^{-4t}
\end{equation}
for some odd $s$ and $t$. We will determine $d, s$ and $t$ up to group automorphisms.

By Lemma~\ref{structurenew}(ii), $[a^4,b^2]=[a^2, b^4]=1$; since $d-1$ is even, this implies that
$$\begin{array}{lll}
c^{a^2b^{-2}}&=&b^2a^{-2}a^{2d}b^{-2d}a^2b^{-2}
=b^2a^{2(d-1)}b^{-2d}a^2b^{-2}\\
&=&a^{2(d-1)}b^{-2(d-1)}a^2b^{-2}
=a^{2d}b^{-2d}=c,
\end{array}$$
\f so  $c^{a^2}=c^{b^2}$. Since
$$c^{a^2}=(c^a)^a=(c^sa^{4t})^a=(c^a)^sa^{4t}=c^{s^2}a^{4t(s+1)}$$
\f and
$$c^{b^2}=(c^b)^b=(c^sb^{-4t})^b=(c^b)^sb^{-4t}=c^{s^2}b^{-4t(s+1)},$$
we see that $a^{4t(s+1)}=b^{-4t(s+1)}$. However, $A\cap B=1$,
so $a^{4t(1+s)}=1$ and hence $c^{a^2}=c^{s^2}$. Since
$t$ is odd, we have $s\equiv -1\ (\mod 2^{e-2})$. In what
follows, we set $s=-1+l2^{e-2}$; since $|c|=2^{e-1}$ we can assume that $l=0$ or $1$. Then $s^2\equiv 1\ (\mod
2^{e-1})$, and because $|c|=2^{e-1}$ we have
$c^{a^2}=c^{s^2}=c$. Thus $a^2$ commutes with $c=a^{2d}b^{-2d}$, and hence with $b^{2d}$; since $d$ is odd we therefore have
\begin{equation}
[a^2, b^2]=1.
\end{equation}

\f Using equation~(6) we see that
\begin{equation}
\begin{array}{lll}
b^{-1}a^{2j}b&=&((a^b)^2)^j=((ac^{-1})^2)^j=(ac^{-1}ac^{-1})^j=(a^2(c^a)^{-1}c^{-1})^j\\
&=&(a^2a^{-4t}c^{-(s+1)})^j=a^{2(1-2t)j}c^{-j(s+1)}
\end{array}
\end{equation}
for each positive integer $j$. By taking
$j=2^{e-2}$ we deduce that the involution $a^{2^{e-1}}$ is central in $G$, and the same holds for $b^{2^{e-1}}$. In what
follows we set $z=a^{2^{e-1}}b^{2^{e-1}}=c^{2^{e-2}}$.
By equations (6) and (8) we have
\begin{equation}
c^a=c^{-1}a^{4t}z^l, \quad c^b=c^{-1}b^{-4t}z^l,
  \quad b^{-1}a^{2j}b=a^{2(1-2t)j}z^{lj}.
\end{equation}
\f From $c^b=(a^{2d}b^{-2d})^b$ and equation~(9) we
have
$$
a^{-2d}b^{2d}b^{-4t}z^l=a^{2(1-2t)d}z^{ld}b^{-2d},$$
\f so
$$a^{4(1-t)d}=b^{4(d-t)}$$
\f and hence
$$(1-t)d\equiv d-t\equiv 0\ (\mod 2^{e-2}).$$
\f Solving these equations gives
$$d\equiv t\equiv 1\ (\mod 2^{e-2}).$$
\f Writing $d=1+k2^{e-2}$ where $k=0$ or $1$, and using $b^{2^e}=1$, we see from these two congruences that the relations $c=a^{2d}b^{-2d}$,
$c^a=c^sa^{4t}$ and $c^b=c^sb^{-4t}$ can be respectively
rewritten as
\begin{equation}
[b, a]=a^{2+k2^{e-1}}b^{-2+k2^{e-1}},\,
(b^2)^a=a^{l2^{e-1}}b^{-2+l2^{e-1}},\,
(a^2)^b=a^{-2+l2^{e-1}}b^{l2^{e-1}}
\end{equation}
 \f where $k,
l\in\{0, 1\}$. By combining the relations in (10) with the fact
that $a^{2^e}=b^{2^e}=[a^2, b^2]=1$ we see that $G$ satisfies all the defining relations of $G_2(e;k,l)$ in (2). Thus $G$ is an epimorphic image of $G_2(e; k,l)$, and since these two groups have the same order, they are isomorphic.

If $e=2$ then $c=a^2b^{-2}=a^2b^2$ is an involution
commuting with both $a$ and $b$, so $c$ is central in $G$. Thus
$[a,b^2]=[b,a^2]=1$ and it follows that $G\cong G_2(2;0,0)$.
\end{proof}

Recall that Theorem~\ref{BJ87.19}, of Berkovich and Janko, states that a $2$-generator $2$-group with exactly one non-metacyclic maximal subgroup, and with a derived group isomorphic to $C_{2^r}\times C_{2^{r+1}}$ for some $r\ge 2$, has a presentation of the form~(4). Here we consider a subset of these groups, namely those for which $x^2$ and $w$ are powers of the central involution $z$. For each $r\ge 2$, and for each pair $k, l\in\{0, 1\}$, let $G=G(k,l)$ denote the group given by the presentation~(4) with $x^2=z^k$ and $w=z^l$, that is,
\begin{equation}
\begin{array}{ll}
G(k,l)=\lg a,x \mid
    &a^{2^{r+2}}=1, \ [a,x]=v, \ [v,a]=b, \ v^{2^{r+1}}=b^{2^r}=[v,b]=1,\\
    &v^{2^r}=z, \ b^{2^{r-1}}=u, \ x^2=z^k, \ b^x=b^{-1},\\
    &v^x=v^{-1}, \ b^a=b^{-1}, \ a^4=v^{-2}b^{-1}w, \ w=z^l\, \rg.
\end{array}
\end{equation}
Our aim is to show that $G$ is isomorphic to the group $G_2=G_2(e;k,l)$, where $e=r+2$. 
To avoid notational confusion, let us present $G_2$ as
\begin{equation}
\begin{array}{ll}
G_2(e;k,l)=\lg \ a_1, b_1 \di& a_1^{2^e}=b_1^{2^e}=[b_1^2, a_1^2]=1,\
[b_1,a_1]=a_1^2b_1^{-2}(a_1^{2^{e-1}}b_1^{2^{e-1}})^k,\\\
&(b_1^2)^{a_1}=b_1^{-2}(a_1^{2^{e-1}}b_1^{2^{e-1}})^l,\
(a_1^2)^{b_1}=a_1^{-2}(a_1^{2^{e-1}}b_1^{2^{e-1}})^l\ \rg .
\end{array}
\end{equation}

\begin{theorem}\label{isomorphism} For each $e=r+2\ge 4$ there is an isomorphism from $G(k,l)$ to $G_2(e;k,l)$ sending the generators $a$ and $x$ of $G(k,l)$ to $a_1$ and $a_1^{-1}b_1$ in $G_2(e;k,l)$.
\end{theorem}

\begin{proof} 
We first show that  the map $a\mapsto a_1,\ x\mapsto a_1^{-1}b_1$ extends to a homomorphism $G\to G_2$. We map the other elements of $G$ appearing in~(11) into $G_2$ by
$$v\mapsto a_1^{-2}b_1^2z_1^k,\quad b\mapsto b_1^{-4}z_1^l,
\quad  z\mapsto z_1,\quad u\mapsto b_1^{2^{e-1}}\quad{\rm and}
\quad w\mapsto z_1^l$$
where $z_1=a_1^{2^{e-1}}b_1^{2^{e-1}}$. We need to show that the defining relations for $G$ in (11) are satisfied when $a, x, v, b, z, u$ and $w$ are replaced with their images in $G_2$. This is a routine matter, using the properties of $G_2$ proved in Section~2, so we will simply illustrate it in a typical case, namely the relation $x^2=z^k$. For this we need to show that $(a_1^{-1}b_1)^2=z_1^k$ in $G_2$. Using Lemma~\ref{G3}(ii) and the fact that $z_1$ is in the centre of $G_2$ we have
$$(a_1^{-1}b_1)^2=a_1^{-1}(b_1a_1^{-1})b_1
=a_1^{-1}(a_1b_1^{-1}z_1^k)b_1=z_1^k,$$
as required. The other cases are similar, so the mapping extends to a homomorphism $\theta: G\to G_2$. This is an epimorphism since $a_1$ and $a_1^{-1}b_1$ generate $G_2$.

We proved in~\cite[Prop.~2.1]{DJKNS2} that $|G_2|=2^{2e}$, so $|G|\ge 2^{2e}$.The defining relations
$$v^{2^{r+1}}=b^{2^r}=[v,b]=1$$
for $G$ show that $\lg v, b\rg$ has order at most $2^{2r+1}$. The relations
$$x^2=z^k \ (=v^{2^rk}), \   v^x=v^{-1},  \ b^x=b^{-1},$$
show that $\lg v, b\rg $ is a normal subgroup of index at most $2$ in $\lg v, b, x \rg$, so the latter group has order at most $2^{2r+2}$. Finally the relations
$$a^4=v^{-2}b^{-1}z^l, \ b^a=b^{-1}, \ v^a=vb, \ x^a=xv^{-1}$$
show that $\lg v, b, x\rg$ is a normal subgroup of index at most $4$ in $\lg v, b, x, a\rg=G$, so $|G|\le 2^{2r+4}=2^{2e}$. Thus $|G|=|G_2|$, so $\theta$ is an isomorphism. \end{proof}

This confirms the assertions in~\cite{BJ, Jan} that the groups $G(k,l)$ have order $2^{2r+4}$, a fact which is not immediately apparent from the presentation~(11).

\section{Regular embeddings of $K_{n,n}$ where $n=2^e$}

A \textit{map\/} $\mathcal M$ is a cellular embedding of a connected graph $K$ in a closed orientable surface. It is \textit{(orientably) regular} if the group $\Aut(\mathcal{M})$ of all orientation-preserving automorphisms of the embedding acts regularly on the oriented edges (darts) of $K$.

It was shown in \cite[Section~2]{JNS1} that every regular
embedding $\mathcal{M}$ of a complete bipartite graph $K_{n,n}$ determines an $n$-isobicyclic triple $(G,a,b)$. Here $G$ is the subgroup
$\Aut_0(\mathcal{M})$ of index $2$ in $\Aut(\mathcal{M})$ leaving
the bipartition of $K_{n,n}$ invariant. The generators $a$ and
$b$ rotate a chosen edge $e=uv$ around its incident vertices $u$ and $v$ to the next edge, following the orientation of ther surface around $u$ or $v$. The automorphism of $G$ transposing $a$ and $b$ is induced by conjugation by the map automorphism reversing $e$. Conversely, every $n$-isobicyclic triple $(G,a,b)$ arises in this way, with $(G_1,a_1,b_1)$ and $(G,a,b)$ giving isomorphic maps if and only if there is an isomorphism $G_1\to G$ sending $a_1$ to $a$ and $b_1$ to $b$ (see~\cite{JNS1} or~\cite[Proposition 2]{JNS2}). Thus an isobicyclic group $G$ may have inequivalent pairs $a, b$ leading to non-isomorphic maps.

The following characterisation of regular embeddings of $K_{n,n}$, where $n=2^e$ and $\Aut_0(\mathcal{M})$ is non-metacyclic, was proved in~\cite{DJKNS2}. Here we give a different proof, using the structure of non-metacyclic isobicyclic $2$-groups
described in earlier sections.

\begin{theorem}\label{rl}
For each $n=2^e\geq 8$ there are exactly four non-isomorphic
regular embeddings $\mathcal{M}$ of $K_{n,n}$ for which
$\Aut_0(\mathcal{M})$ is non-metacyclic; these correspond to
the four isobicyclic triples $(G, a,b)$, where $G=G_2(e;
k,l)$ and $k,l\in\{0,1\}$. There is
exactly one regular embedding $\mathcal{M}$ of $K_{4,4}$
for which $\Aut_0(\mathcal{M})$ is non-metacyclic; this map
corresponds to the isobicyclic triple $(G, a,b)$ where
$G=G_2(2;0,0)$.
\end{theorem}

\begin{proof}
If $e=2$ the result follows directly from
Theorem~\ref{thm:nonmetachar}(i). We may therefore assume that $e\ge 3$, so by Theorem~\ref{thm:nonmetachar}(ii) there are at most four isomorphism classes of isobicyclic triples, corresponding to the four presentations
$G_2(e;k,l)$ where $k,l\in\{0,1\}$. By Corollary~\ref{iso}, the groups
$G_2(e; 0, 0)$, $G_2(e; 1, 0)$ and $G_2(e; 0, 1)$ are mutually
non-isomorphic, and hence so are the corresponding isobicyclic triples. To complete the classification it is enough to
show that the triples corresponding to the isomorphic groups
$G_2(e;0,1)$ and $G_2(e;1,1)$ are not equivalent. If there is an isomorphism from $G_2(e;1,1)=\lg a_1\rg \lg b_1\rg
$ to $G_2(e; 0, 1)=\lg a\rg \lg b\rg$ taking $a_1$ to $a$ and
$b_1$ to $b$ then condition~(1) of Lemma~\ref{auto} gives $1=k_1\equiv k=0\pmod{2}$, a
contradiction. Hence there are four non-isomorphic maps, as
claimed.
\end{proof}

\vskip 1cm
\begin{center}
{\bf Acknowledgements}
\end{center}
\vskip -.8mm {\small The authors acknowledge the
support of  the following grants.  The first author was partially supported by grants NNSF(10971144) and BNSF(1092010).
The fourth and the fifth authors acknowledge partial support
from the grants APVV-0223-10, VEGA 1/1085/11, and from the grant
APVV-ESF-EC-0009-10 within the EUROCORES Programme EUROGIGA (project GReGAS) of the European Science Foundation.
The authors are also grateful to the organisers of workshops at Capital National University, Beijing, and the Fields Institute, Toronto, where much of this research took place.

\end{document}